\documentclass[12pt]{amsart}
\usepackage[utf8]{inputenc}
\usepackage[english]{babel}

\usepackage{amsmath,amsthm,amsfonts,mathrsfs,amssymb}
\usepackage{bm}
\usepackage{cite}

\newtheorem{theorem}{Theorem}[section]
\newtheorem{lemma}[theorem]{Lemma}
\newtheorem{corollary}[theorem]{Corollary}
\newtheorem{remark}[theorem]{Remark}

\newtheorem{definition}[theorem]{Definition}

\newenvironment{proposition*}[1]{\smallskip\noindent{\bf #1.}\it}{\medskip}
\newenvironment{theorem*}[1]{\smallskip\noindent{\bf #1.}\it}{\medskip}

\numberwithin{equation}{section}
\newcommand\myIm{\operatorname{Im}}

\newcommand{\wt}{\widetilde}

\newcommand{\bC}{\mathbb C}
\newcommand{\bN}{\mathbb N}
\newcommand{\bR}{\mathbb R}

\newcommand{\cA}{\mathcal A}
\newcommand{\cB}{\mathcal B}
\newcommand{\cC}{\mathcal C}
\newcommand{\cD}{\mathcal D}
\newcommand{\cI}{\mathcal I}
\newcommand{\cQ}{\mathcal Q}

\newcommand{\sX}{\mathscr X}

\newcommand{\bol}{\boldsymbol \lambda}
\newcommand{\bom}{\boldsymbol \mu}
\begin{document}
	
\title[Trace Formula for Schr\"odinger  Operators]%
{On the First Trace Formula  for Schr\"odinger  Operators}%

\author[R.~Hryniv \and Ya.~Mykytyuk]%
{Rostyslav Hryniv \and Yaroslav Mykytyuk}%

\address[R.H.]{
	Ukrainian Catholic University, 2a Kozelnytska str., 79026, Lviv, Ukraine \and
	University of Rzesz\'{o}w, 1 Pigonia str., 35-310 Rzesz\'{o}w, Poland
}%
\email{rhryniv@ucu.edu.ua, rhryniv@ur.edu.pl}

\address[B.M.]{Lviv National University,
	1 Universytetska st., 79602 Lviv, Ukraine}%
\email{bohdmelnyk@gmail.com}%

\address[Ya.M.]{Lviv National University,
	1 Universytetska st., 79602 Lviv, Ukraine}%
\email{yamykytyuk@yahoo.com}%

\thanks{}%

\date{2 March 2019}%


\begin{abstract}
  We prove that the so-called first trace formula holds for all Schr\"odinger operators on the line with real-valued integrable potentials.
\end{abstract}

\maketitle

\section{Introduction}\label{sec:intr}

For a real-valued function~$q$ that is integrable on the line, we denote by~$T_q$ the self-adjoint Schr\"odinger operator in the Hilbert space $L_2(\mathbb{R})$ given by the differential expression
\begin{equation*}
\mathfrak{t}_q(f) := -f'' + qf 
\end{equation*}
on the maximal domain. It is well known (see e.g., \cite[Ch.~15]{Wei}, \cite{Sto}) that for every such~$q$ the absolutely continuous spectrum of~$T_q$ coincides with the non-negative half-line~$\mathbb{R}_+$ and there is no singular continuous spectrum. Moreover, to every non-zero $k^2 \in \mathbb{R}_+$ there correspond the right and left Jost solutions of the equation $\mathfrak{t}_q(f) = k^2 f$ that are asymptotic to $e^{\pm ikx}$ at $\pm\infty$ respectively~\cite{DT}. These Jost solutions allow one to introduce the right and left reflection coefficients $r_\pm(\cdot;q)$, which turn out to be continuous functions on~$\mathbb{R}_0:= \mathbb{R}\setminus\{0\}$ (see the next section).  

Next, the discrete part of the spectrum of $T_q$ consists of at most countably many simple negative eigenvalues with the only possible accumulation point at the origin. We denote these eigenvalues by 
\[
-\kappa_1^2< \dots < -\kappa_N^2 <0
\]
with $N=0$ when the discrete spectrum is empty and $N=\infty$ when it is countable. It will be convenient to assume that $N=\infty$ by formally setting $\kappa_{N+1} = \kappa_{N+2} = \dots =0$ otherwise. In this way every real-valued summable potential $q$ generates a sequence $\kappa(\cdot;q): \bN \to [0,\infty)$ with the following three properties: 
\begin{enumerate}
	\item $\kappa(j;q) > 0$ implies that~$\kappa(j;q) > \kappa(j+1;q)$;
	\item if $\kappa(j;q) =0$, then $\kappa(j+1;q)=0$;
	\item  the set $\{ -\kappa^2(j;q) \mid j\in\bN, \kappa(j;q)>0\}$ coincides with the negative spectrum of the operator~$T_q$. 
\end{enumerate}

The pioneering paper of Gardner, Green, Kruskal, and Miura~\cite{GGKM} of 1967 on solvability of the Korteweg--de Vries (KdV) equation by the method of inverse scattering transform for the related family of the Schr\"odinger operators motivated the increasing interest in spectral theory of the latter. Later, they discovered~\cite{MGK} that the KdV equation possesses a series of conservation laws and constants of motions, and then Faddeev and Zakharov in their famous paper~\cite{FZ} of 1971 interpreted KdV as a Hamiltonian system and showed that the conservation laws are the first integrals of that system. Also, they expressed the conserved quantities of the KdV equation via the spectral characteristics of the Schr\"odinger operators in the form of so called trace formulae. The first of these trace formulae reads
\begin{equation}\label{eq:intr.FZ}
-4\sum_{j=1}^{N} \kappa(j;q) 
+ \frac{1}{\pi}\int_{\mathbb{R}}\log (1-|r_\pm(k;q)|^2)^{-1}\,dk 
= \int_{\mathbb{R}} q(x)\,dx
\end{equation}
and thus relates the potential~$q$ of the Schr\"odinger operator $T_q$ and its scattering characteristics---eigenvalues and the reflection coefficients. Although the initial proof was for real-valued $q$ from the Schwartz class of infinitely smooth and rapidly decaying functions, in their paper~\cite{GH} Gesztesy and Holden showed, among many other important results, that~\eqref{eq:intr.FZ} holds also if $q$ is a real-valued potential belonging to the space~$L_1(\mathbb{R},(1+|x|)^\varepsilon dx)$ for some positive~$\varepsilon$. Our main aim in this note is to prove that the first trace formula takes place for all real-valued integrable potentials; this result will be used in our further research of reflectionless potentials and generalized soliton solutions of KdV.

In what follows, we shall denote by $\cQ_1$ the set of all real-valued functions in~$L_1(\mathbb{R})$ considered as a real Banach space with the $L_1$-norm. Then, in virtue of the Lieb--Thirring inequality~\cite{LT} (with the case under consideration proved by Weidl~\cite{Weidl} and the exact constant established by Hundertmark~a.o.~\cite{HLT}), we have
\begin{equation}\label{eq:intr.LT}
\sum_{j=1}^{\infty}\kappa(j;q) \le \frac12\|q\|_{L_1}, \qquad q\in \cQ_1,
\end{equation}
i.e., $\kappa(\cdot;q) \in \ell_1(\bN)$ for every potential~$q\in\cQ_1$. Observe also that $|r_\pm(k;q)|<1$ on~$\mathbb{R}_0$ for such $q$, whence the function
\begin{equation*}\label{eq:intr.h}
h(\cdot; q):= \log(1-|r_\pm(k;q)|^2)^{-1}
\end{equation*}
is non-negative a.e.\ on~$\mathbb{R}$. Combining~\eqref{eq:intr.LT} and \eqref{eq:intr.FZ}, we conclude that under the additional assumptions of the paper~\cite{GH}, the function $h(\cdot;q)$ belongs to the space~$\cQ_1$ and, moreover,
\[
\|h(\cdot;q)\|_{L_1} \le 3\pi \|q\|_{L_1}. 
\]

Our main results are given by the following theorem:

\begin{theorem}\label{thm:main}
	The mappings 
	\begin{equation}\label{eq:intr.cont}
	\cQ_1 \ni q \mapsto \kappa(\cdot;q) \in \ell_1(\bN), \qquad 
	\cQ_1 \ni q \mapsto h(\cdot; q) \in \cQ_1
	\end{equation}	
	are continuous; moreover, the first trace formula~\eqref{eq:intr.FZ} holds true for all $q \in \cQ_1$.
\end{theorem}

The paper is organized as follows. In Section~\ref{sec:Jost}, we study properties of the Jost solutions of the equation~$-y''+qy=z^2 y$ for complex~$z$ and establish behaviour of the scattering coefficient~$a$ near the origin. In Section~\ref{sec:aux} some special properties of the Blaschke products and Cauchy integrals are established and, finally, in Section~\ref{sec:proof} we give the proof of Theorem~\ref{thm:main}.


\subsection*{Notations.} Throughout the paper, we denote by $\bC_+$ the open upper-half complex plane and by $\overline{\bC}_+$ its closure. The symbols $\pm$ appearing in a statement are meant to denote two separate statements, one with the sign $+$ and the other with the sign~$-$. 
As usual, $\|\cdot\|_1$ is the norm of~$L_1(\mathbb{R})$; recall that $\cQ_1$ will denote the subset of all real-valued functions in~$L_1(\mathbb{R})$ considered as a real Banach space with norm~$\|\cdot\|_1$.
Finally, to simplify the notations, we shall denote by $r=r(\cdot;q)$ the right reflection coefficient $r_+(\cdot;q)$ corresponding to $q\in L_1(\mathbb{R})$.


\section{Jost solutions}\label{sec:Jost}

Existence of the Jost solutions for the Schr\"odinger operator~$T_q$ along with some useful estimates was established e.g.\ in~\cite{DT}; note that although the authors imposed therein more restrictive assumptions on the potentials, only integrability of~$q$ was used to derive some of the results. We collect them here since we need slightly more general formulations to study continuous dependence on the potential~$q$, in order to set up the notations needed throughout the paper, and for the convenience of the reader. 

\begin{definition}\label{def:Jost.Jost}
	Assume that~$q\in L_1(\bR)$ and $\lambda \in \overline{\bC}_+$.  Solutions  $e_+(\cdot,\lambda)=e_+(\cdot,\lambda;q)$ and $e_-(\cdot,\lambda)=e_-(\cdot,\lambda;q)$ of the equation
	\begin{equation}\label{eq:Jost.eq}
	-f'' +qf=\lambda^2 f
	\end{equation}
	are called respectively its \emph{right} and \emph{left} \emph{Jost solutions} if they satisfy the asymptotic relations
	\begin{equation*}\label{eq:Jost.asymp}
	e^{\mp i\lambda x} e_\pm(x,\lambda) = 1 +o(1),
	\qquad x\to \pm\infty.
	\end{equation*}
\end{definition}

\begin{remark}\label{rem:Jost.left}
	Assume that $q\in L_1(\bR)$ and that $q^\sharp$ satisfies the relation $q^\sharp(x)=q(-x)$ for $x\in\bR$. It then follows from definition that $e_-(x,\lambda; q) = e_+(-x,\lambda; q^\sharp)$. For this reason, it only suffices to study the right Jost solutions. To simplify the notations, we shall write $e(x,\lambda)$ instead of $e_+(x,\lambda;q)$ whenever no confusion can arise.
\end{remark}

As in~\cite{DT}, one shows that the function $m(x,\lambda):= e^{-i\lambda x}e(x,\lambda)$ satisfies the differential equation
\[
m'' + 2i\lambda m'  = q(x)m 
\] 
as well as the integral equation
\begin{equation}\label{eq:Jost.D}
m(x,\lambda) = 1 + \int_{x}^\infty D(t-x,\lambda)q(t)m(t,\lambda)\,dt
\end{equation}
with
\[
D(y,\lambda) := \frac{1}{2i\lambda}\bigl(e^{2i\lambda y} - 1\bigr)
\]
for $y>0$ and $D(y,\lambda)=0$ otherwise.
Thus to construct the Jost solution~$e(\cdot,\lambda)$, it is sufficient to solve the above integral equation for~$m$.

Denote by~$\sX$ the Banach space~$L_\infty(\bR)$ with the standard supremum norm~$\|\cdot\|_\infty$ and by~$\cB(\sX)$ the Banach algebra of all linear continuous operators acting in~$\sX$. Next, set 
\[
\Omega_\varepsilon :=\{z\in \overline{\bC}_+ \mid  |z|>\varepsilon\},  \qquad \varepsilon\ge 0, 
\]	
and introduce for $q\in L_1(\bR)$ and $\lambda\in \Omega_\varepsilon$ an integral operator $\cD_{q,\lambda}$ in~$\sX$ via
\begin{equation*}\label{eq.23}
(\cD_{q,\lambda}f)(x)=\int_x^\infty D(t-x,\lambda)q(t)f(t)\, dt, \qquad x\in\bR.
\end{equation*}

\begin{lemma}\label{le.22}  
	For every $q\in L_1(\bR)$ and $\lambda\in\Omega_0$, the operator $\cD_{q,\lambda}$ is quasi-nilpotent in~$\cB(\sX)$ and
	\begin{equation}\label{eq.24}
	\|\cD^n_{q,\lambda}\|\le \frac1{n!} \left(\frac{\|q\|_{1}}{|\lambda|}\right)^n, \qquad n\in \bN.
	\end{equation}
	Moreover, for a fixed $q\in L_1(\bR)$ the function
	$	\Omega_0\ni\lambda\mapsto \cD_{q,\lambda} \in \cB(\sX)  $
	is analytic in~$\bC_+$ and continuous in~$\Omega_0$.
\end{lemma}

\begin{proof}
	Since
	\begin{equation*}
	\left| D(y,\lambda)\right|\le |\lambda|^{-1}, \qquad y\ge0, \quad \lambda\in\Omega_0,
	\end{equation*}
	for every $f\in \sX$, $q\in L_1(\bR)$, and $\lambda\in\Omega_0$ we get
	\begin{equation}\label{eq:Jost.D-bound}
	|\cD_{q,\lambda}f(x)| \le \frac{\|f\|_\infty}{|\lambda|} \int_x^\infty |q(t)|\,dt,
	\end{equation}
	so that 
	\begin{equation}\label{eq:Jost.Dn-bound}
	\|\cD^n_{q,\lambda}\|
	\le\int_{t_n\le \dots\le t_1} |\lambda|^{-n}
	|q(t_1)|\dots|q(t_n)| \,dt_1\dots dt_n=\frac1{n!} \left(\frac{\|q\|_{1}}{|\lambda|}\right)^n .
	\end{equation}
	Therefore, $\cD_{q,\lambda}$ is quasi-nilpotent.
	
	Taking into account that for every $\varepsilon>0$  there exists $C_\varepsilon>0$ such that 
	\[   
	\Bigl|\frac{d}{d\lambda} D(y,\lambda)\Bigr|\le C_\varepsilon
	\]
	for all $y\ge 0$ and all $\lambda\in \Omega_\varepsilon$, we conclude that the function $\lambda\mapsto \cD_{q,\lambda}$ is analytic in~$\bC_+$.
	Finally, continuity of the mapping $\lambda \mapsto D(\cdot,\lambda)$ from $\Omega_0$ onto $\sX$ implies that the function~$\lambda\mapsto \cD_{q,\lambda}$ is continuous on~$\Omega_0$ by virtue of the estimate similar to~\eqref{eq:Jost.D-bound}.
\end{proof}

Denote by $\cI$ the identity operator in~$\cB(\sX)$; the above lemma implies that the operator $\cI-\cD_{q,\lambda}$ is boundedly invertible in~$\cB(\sX)$ for suitable $q$ and $\lambda$. Set 
\begin{equation}\label{eq.26}
m_q(\cdot,\lambda):=(\cI-\cD_{q,\lambda})^{-1}m_0,    \qquad q\in L_1(\bR), \quad \lambda\in\Omega_0,
\end{equation}
with $m_0\equiv 1 \in\sX$; as we mentioned at the beginning of this section, the function~$m_q$ is related to the right Jost solution $e(\,\cdot\,,\lambda)$. We start by establishing some bounds on~$m_q$.

\begin{lemma}\label{le.23}  Assume that $q\in L_1(\bR)$. Then the function
	$\Omega_0\ni\lambda\mapsto m_q(\cdot,\lambda)\in \sX$
	is analytic in the half-plane~$\bC_+$ and continuous on~$\Omega_0$. In particular, for every $\lambda\in \Omega_0$
	\begin{subequations}\label{eq.27}
	\begin{align}
	\|m_q(\cdot,\lambda)\|_\infty\le &   \exp\{\|q\|_{1}/|\lambda|\},    \\
	\|m_q(\cdot,\lambda)-1\|_\infty\le & \frac{\|q\|_{1}}{|\lambda|} \exp\{\|q\|_{1}/|\lambda|\}.                                                                                                                        
	\end{align}
	\end{subequations}
	Moreover, if $q,\wt q\in L_1(\bR)$ and $\|q\|_1,\|\wt q\|_1 \le \alpha\,\,(\alpha>0)$, then
	\begin{equation}\label{eq.28}
	\|m_q(\cdot,\lambda)-m_{\wt q}(\cdot,\lambda)\|_\infty\le \exp\{(2\alpha/|\lambda|)\} \frac{\|q-\wt q\|_{1}}{|\lambda|}, \qquad \lambda\in \Omega_0.
	\end{equation}
\end{lemma}

\begin{proof} The fact that the function $\lambda\mapsto m_q(\cdot,\lambda)$ is analytic in $\bC_+$
	and continuous in $\Omega_0$ follows from Lemma~\ref{le.22}. Next, on account of~\eqref{eq.24}
	and the equality $(\cI-\cD_{q,\lambda})^{-1}=\sum_{n=0}^\infty \cD_{q,\lambda}^n$, one gets
	\begin{equation}\label{eq.29}
	\|(\cI-\cD_{q,\lambda})^{-1}\|
	\le \sum_{n=0}^\infty \frac1{n!} \left(\frac{\|q\|_1}{|\lambda|}\right)^n
	= \exp\{\|q\|_1/|\lambda|\},
	\end{equation}
	which in view of~\eqref{eq.26} produces the bounds \eqref{eq.27} in a straightforward manner.
	
	Now the second resolvent identity
	\[   
	(\cI-\cD_{q,\lambda})^{-1}-(\cI-\cD_{\wt q,\lambda})^{-1}
	= (\cI-\cD_{q,\lambda})^{-1}(\cD_{q,\lambda}-\cD_{\wt q,\lambda} )(\cI-\cD_{\wt q,\lambda})^{-1}
	\]
	on account of~\eqref{eq.26}, \eqref{eq:Jost.Dn-bound}, and \eqref{eq.29} results in the estimate
	\begin{multline*}
	\|m_q(\cdot,\lambda)-m_{\wt q}(\cdot,\lambda)\|_\infty
	\le \|(\cI-\cD_{q,\lambda})^{-1}\| \,\|(\cI-\cD_{\wt q,\lambda})^{-1}\|\, 
	\|\cD_{q-\wt q ,\lambda}\| \\
	\le \exp\{(2\alpha/|\lambda|)\}\|\cD_{q-\wt q,\lambda}\|
	\le \exp\{(2\alpha/|\lambda|)\} \frac{\|q-\wt q\|_1}{|\lambda|}.
	\end{multline*}
	The proof is complete.
\end{proof}

\begin{corollary}\label{cor:Jost.e}
	For every $q \in L_1(\bR)$ and $\lambda \in \Omega_0$, the function $e^{i\lambda x}m_q(x,\lambda)$ is the right Jost solution $e(x,\lambda;q)$ of equation~\eqref{eq:Jost.eq}. Moreover, as $x\to+\infty$,
	\[
	e'(x,\lambda;q)  = i\lambda e^{i\lambda x} + o(1).
	\]
\end{corollary}

\begin{proof}
	In view of~\eqref{eq:Jost.D}, the function $e^{i\lambda x}m_q(x,\lambda)=:h(x)$ for $\lambda\in\Omega_0$ satisfies the equality
	\begin{equation}\label{eq.215}
	h(x) = e^{i\lambda x} + \int_x^\infty \frac{\sin{\lambda(t-x)}}{\lambda}q(t)h(t)\, dt, 
	\qquad x\in\bR.
	\end{equation}
	Direct verification shows that $h$ solves~\eqref{eq:Jost.eq} and obeys the required asymptotics and thus is the right Jost solution $e(x,\lambda;q)$ of~\eqref{eq:Jost.eq}. 	Differentiating~\eqref{eq:Jost.D} in $x$, we find that 
	\[
	m'_q(x,\lambda) = -\int_x^\infty e^{2i\lambda(t-x)} q(t) m_q(t,\lambda)\,dt = o(1) 
	\]
	as $x\to+\infty$. Therefore, 
	\[
	e'(x,\lambda; q) - i\lambda e(x,\lambda; q) = e^{i\lambda x} m'_q(x,\lambda) = o(1)
	\]
	as $x \to+\infty$, which completes the proof.
\end{proof} 

Denote by $\cA$ the set of all functions that are analytic in $\bC_+$ and continuous and bounded in $\Omega_\varepsilon$ for every~$\varepsilon>0$. Further, let $\cC$ denote the set of all continuous complex-valued functions on 
$\bR_0:=\bR\setminus\{0\}$ that are bounded on $\bR_\varepsilon:=\bR\setminus(-\varepsilon,\varepsilon)$ for every $\varepsilon>0$. 
The sets $\cA$ and  $\cC$ are endowed with the topology of uniform convergence on compact sets of $\Omega_0$ and $\bR_0$ respectively.

\begin{lemma}\label{lem:Jost.a} 
	For every $q\in L_1(\bR)$ and $\lambda\in \bR_0$, 
	\begin{equation}\label{eq.211}
	e(x,\lambda;q)=  a(\lambda,q) e^{i\lambda x}+ b(\lambda,q) e^{-i\lambda x}+o(1),  \qquad x\to -\infty,
	\end{equation}
	where
	\begin{equation}\label{eq.212}
	a(\lambda,q):=1 - \frac1 {2i\lambda}\int_{-\infty}^\infty q(t)m_q(t,\lambda)\, dt,  \qquad \lambda\in \Omega_0,
	\end{equation}
	and
	\begin{equation}\label{eq.213}
	b(\lambda,q):=\frac1 {2i\lambda} \int_{-\infty}^\infty
	e^{2i\lambda t}
	q(t)m_q(t,\lambda)\, dt , \qquad \lambda\in \bR_0.
	\end{equation}
	Moreover, formulae \eqref{eq.212} and \eqref{eq.213} define continuous mappings
	\begin{equation}\label{eq.214}
	L_1(\bR)\ni q \mapsto a(\cdot,q)\in \cA, \qquad   L_1(\bR)\ni q \mapsto b(\cdot,q)\in \cC.
	\end{equation}
\end{lemma}

\begin{proof} 
	Rewriting the equality \eqref{eq.215} in the form
	\begin{equation*}
	e(x,\lambda;q)
	= e^{i\lambda x} - \frac{e^{i\lambda x}} {2i\lambda}\int_x^\infty q(t)m_q(t,\lambda)\, dt 
	+ \frac{e^{-i\lambda x}} {2i\lambda} \int_x^\infty e^{2i\lambda t}  q(t)m_q(t,\lambda)\, dt,
	\end{equation*}
	we arrive at~\eqref{eq.211}.
	
	Continuity of the mappings of~\eqref{eq.214} is a direct corollary of Lemma~\ref{le.23}.
\end{proof}

\begin{remark}\label{re.25}
	As is known from the classical scattering theory \cite[Ch.~3.5]{Mar}, for potentials~$q$ from the Faddeev--Marchenko class
	\begin{equation*}\label{eq.216}
	\cQ_{1,1}:=\Bigl\{q\in\cQ_1 \mid \int_{\bR}(1+|x|)|q(x)|\,dx <\infty \Bigr\}
	\end{equation*}
	we have the identity
	\begin{equation}\label{eq.217}
	|a(\lambda,q)|^2-|b(\lambda,q)|^2=1 , \qquad
	\lambda\in \bR_0.
	\end{equation}
	Since the set $\cQ_{1,1}$ is everywhere dense in $\cQ_1$, continuity of the mappings~\eqref{eq.214} yields the identity~\eqref{eq.217} for all~$q\in\cQ_1$.
	This also implies that the mappings 
	\begin{align*}
	\cQ_1\ni q &\mapsto r_-(\cdot,q):=\frac{b(k,q)}{a(\lambda,q)} \in \cC,  \\ \cQ_1\ni q &\mapsto h(\cdot,q):= \log(1-|r_-(\cdot,q)|^2)\in \cC
	\end{align*}   
	are continuous. 
\end{remark}

As usual, we denote by $\log^+$ the positive part of $\log$, i.e., for positive $x$ we set $\log^+(x) = \log x$ if $\log x>0$ and $\log^+(x) = 0$ otherwise. 
The main result of this section is the following theorem.

\begin{theorem}\label{thm:Jost.a}
	Assume that $q\in \cQ_1$; then 
	\begin{equation}\label{eq:jost.a-asympt}
	\log^+ |a(\lambda;q)|=o(\lambda^{-1}), \qquad \bC_+\ni\lambda \to 0.
	\end{equation}
\end{theorem}

\begin{proof} In view of~\eqref{eq.212}, it suffices to prove that
	\[
	\log^+\{\|m_q(\cdot,\lambda)\|_\infty\} = o(\lambda^{-1}), \qquad \bC_+\ni\lambda \to 0;
	\]
	according to~\eqref{eq.26}, it is enough to show that
	\begin{equation*}\label{eq.221}
	\log^+ \|(\cI-\cD_{q,\lambda})^{-1}\|=o(\lambda^{-1}), \qquad \bC_+\ni\lambda \to 0.
	\end{equation*}
	
	We fix $q\in L_1(\bR)$ and $\lambda\in \bC_+$ and introduce potentials $q_j:=\chi_j q$, $j=1,2,3$, where $\chi_j$ is the characteristic function of the following sets~$A_j$:
	\[
	A_1:=(|\lambda|,+\infty), \qquad  
	A_2:=[-|\lambda|,|\lambda|], \qquad
	A_3:=(-\infty,-|\lambda|).
	\]    
	Observe that $q=q_1+q_2+q_3$, so that 
	\begin{equation*}
	\cD_{q,\lambda}=\cD_{q_1,\lambda}+\cD_{q_2,\lambda}+\cD_{q_3,\lambda}.
	\end{equation*}
	In addition, it follows that 
	\begin{equation*}
	\cD_{q_1,\lambda}\,\cD_{q_2,\lambda}=\cD_{q_1,\lambda}\,\cD_{q_3,\lambda}=
	\cD_{q_2,\lambda}\,\cD_{q_3,\lambda}=0,
	\end{equation*}
	yielding the relations 
	\[
	(\cI-\cD_{q,\lambda})=(\cI-\cD_{q_1,\lambda})(\cI-\cD_{q_2,\lambda})(\cI-\cD_{q_3,\lambda}),
	\]
	and
	\begin{equation}\label{eq.222}
	\|(\cI-\cD_{q,\lambda})^{-1}\|
	\le \|(\cI-\cD_{q_1,\lambda})^{-1}\|\,
	\|(\cI-\cD_{q_2,\lambda})^{-1}\|\,
	\|(\cI-\cD_{q_3,\lambda})^{-1}\|.
	\end{equation}
	Estimate~\eqref{eq.29} gives the bounds 
	\begin{subequations}\label{eq.223}
	\begin{align}
	\|(\cI-\cD_{q_1,\lambda})^{-1}\|&\le \exp\{\|q_1\|_1/|\lambda|\},\\
	\|(\cI-\cD_{q_3,\lambda})^{-1}\|&\le \exp\{\|q_3\|_1/|\lambda|\};
	\end{align}
	\end{subequations}
	but we shall bound the norm~$\|(\cI-\cD_{q_2,\lambda}^{-1})\|$ in a different way. 
	
	Namely, for every $\lambda\in \Omega_0$ and $u\ge0$ we have
	\[
	\left|\frac{e^{2i\lambda u} -1}{2i\lambda}\right| =
	\left|\int_0^u e^{2i\lambda \xi}\, d\xi\right|\le u,
	\]
	which yields the inequality
	\begin{equation*}
	\left| D(t-x,\lambda)\right|\le t-x, \qquad
	x\le t, \quad \lambda\in\Omega_0.
	\end{equation*}
	As a result, for every $n\in\bN$ it holds that
	\[
	\|\cD^n_{q_2,\lambda}\|
	\le C_n(\lambda) \int\limits_{\Pi_{n,\lambda}}  |q(t_1)| \cdots |q(t_n)|\, dt_1\dots dt_n
	\le C_n(\lambda)\frac{\|q\|_1^n}{n!},
	\]
	where
	\[
	\Pi_{n,\lambda}:=\{(t_1,\dots,t_n)\in\bR^n \mid |\lambda|^{-1}\ge t_1\ge \dots t_n\ge -|\lambda|^{-1}|\}	
	\]
	and
	\[  
	C_n(\lambda):= \max_{\Pi_{n,\lambda}}  |\lambda|^{-1}\prod_{j=1}^{n-1} (t_j-t_{j+1}).		
	\]
	Since the geometric mean is less than or equal to the arithmetic one, we conclude that
	\[	
	C_n(\lambda) \le \left(\frac{3|\lambda|^{-1}}{n}\right)^{n},
	\]
	whence
	\begin{equation*}
	\|(\cI-\cD_{q_2,\lambda})^{-1}\|
	\le \sum\limits_{n=0}^\infty \|\cD^n_{q_2,\lambda}\|
	\le 1+ \sum_{n=1}^\infty\frac{(3|\lambda|^{-1}\|q\|_1)^n}{n^{n}n!}.
	\end{equation*}
	Since
	\[    
	(2n)^n n! \ge (2n)! ,  \qquad n\in\bN,
	\]
	we find that  
	\begin{equation}\label{eq.224}
	\|(\cI-\cD_{q_2,\lambda})^{-1}\|
	\le \sum_{n=0}^\infty\frac{(3\sqrt{\|q\|_1/|\lambda|})^{2n}}{(2n)!}
	\le \exp\{3\sqrt{\|q\|_1/|\lambda|}\}.
	\end{equation}
	Taking into account~\eqref{eq.222}, \eqref{eq.223}, and \eqref{eq.224}, we conclude that
	\begin{multline*}
	\log \|(\cI-\cD_{q,\lambda})^{-1}\|
	\le \frac{\|q_1\|_1+\|q_3\|_1}{|\lambda|} 
	+ 3\sqrt{\|q\|_1/|\lambda|}
	\\=\frac1{|\lambda|}\int_{|t|\ge|\lambda|^{-1}}|q(t)|\, dt 
	+3\sqrt{\|q\|_1/|\lambda|}=o(|\lambda|^{-1}), \qquad \bC_+\ni\lambda \to 0.
	\end{multline*}
	The theorem is proved.
\end{proof}


\section{Some auxiliary results}\label{sec:aux}
In this section, we shall establish several results on Blaschke products and Cauchy integrals that will be used in the proof of the main theorem. 

Denote by $\Lambda$ the set of all sequences $\bol=(\lambda_n)_{n\in\bN} \in\ell_1(\bN)$  such that $\lambda_n \in\bC_+\cup\{0\}$ 
for every $n\in\bN$. Every $\bol\in \Lambda $ generates a Blaschke product
\begin{equation*}\label{eq.31}
B(z,\bol):=\prod_{n=1}^\infty \frac{z-\lambda_n}{z-\bar\lambda_n}, \qquad z\in\bC_+.
\end{equation*}
The above product converges absolutely in $\bC_+$ in view of the inequality 
\[
\Bigl|\frac{z-\lambda_n}{z-\bar\lambda_n} - 1 \Bigr| \le \frac{2|\lambda_n|}{|\myIm z|}
\]
and defines there a holomorphic function. 
Observe also that $|B(z,\bol)|\le1$ on $\bC_+$ due to the inequality $|z-\lambda_n| \le |z-\overline{\lambda_n}|$
holding for all $n\in\bN$ and all $z\in\bC_+$.

\begin{lemma}\label{le.31}
	Assume that $\bol=(\lambda_n)_{n\in\bN}\in \Lambda $ and  $B:=B(\cdot,\bol)$. Then
	\begin{equation}\label{eq.32}
	\lim_{\xi\to+0}\int_0^\pi \xi \log|B(\xi e^{it})|^{-1}\,dt=0.
	\end{equation}
\end{lemma}

\begin{proof} Standard arguments show that the quantity
	\[       
	\beta:=\sup_{\lambda\in\bC_+}\int_0^\pi \log{\left|\frac{e^{it}-\bar\lambda}{e^{it}-\lambda}\right|}\, dt
	\]
	is positive and finite. Therefore, for every $\xi>0$ and $j\in\bN$ we have 
	\begin{equation}\label{eq.33}
	\int_0^\pi  \log{\left|\frac{\xi e^{it}-\bar\lambda_j}{\xi e^{it}-\lambda_j}\right|}\, dt\le \beta.
	\end{equation}
	When $|\lambda_j| \le \xi/2$, we get the estimate 
	\[
	\left|\frac{\xi e^{it}-\bar\lambda_j}{\xi e^{it}-\lambda_j}\right|
	\le 1 +\left|\frac{\lambda_j-\bar\lambda_j}{\xi e^{it}-\lambda_j}\right|
	\le 1+\frac{4|\lambda_j|}{\xi},
	\]
	so that
	\begin{equation}\label{eq.34}
	\int_0^\pi \xi \log{\left|\frac{\xi e^{it}-\bar\lambda_j}{\xi e^{it}-\lambda_j}\right|}\, dt
	\le \int_0^\pi 4|\lambda_j|\, dt
	= 4 \pi|\lambda_j|.
	\end{equation}
	Combining~\eqref{eq.33} and \eqref{eq.34}, we conclude that 
	\begin{multline*}
	\int_0^\pi \xi \log|B(\xi e^{it})|^{-1}\,dt
	= \sum_{j=1}^\infty \int_0^\pi \xi \log{\left|\frac{\xi e^{it}-\bar\lambda_j}{\xi e^{it}-\lambda_j}\right|}\,dt\\
	\le \xi\beta \sum_{j \,:\,2|\lambda_j|>\xi} 1  
	+ 4\pi \sum_{j\,:\,2|\lambda_j|\le\xi} |\lambda_j|
	= (\beta + 2\pi) \sum_{j=1}^\infty \min\{\xi,\, 2|\lambda_j|\}=o(1)
	\end{multline*}
	as $\xi\to +0$ by the Lebesgue dominated convergence theorem.  
\end{proof}

\begin{lemma}\label{le.32}
	Assume that $\bol,\bom \in \Lambda$; then the following inequality holds:
	\begin{equation}\label{eq.35}
	|B(z,\bol)-B(z,\bom)|
	\le \frac{2\|\bol-\bom\|_1}{|\myIm z|}, \qquad
	z\in \bC_+.
	\end{equation}
\end{lemma}

\begin{proof} 
	We introduce sequences~$\bol^{(n)}=(\lambda^{(n)}_j)_{j\in\bN}$, $n\in\bN$, via
	\[
	\lambda^{(n)}_j:= \begin{cases}
	\mu_j, & \text{if } j<n;\\
	\lambda_j, & \text{if } j\ge n;
	\end{cases}
	\]
	then
	\begin{equation}\label{eq.37}
	|B(z,\bol)-B(z,\bom)|\le \sum\limits_{n=1}^\infty |B(z,\bol^{(n)})-B(z,\bol^{(n+1)})|, \qquad z\in\bC_+.
	\end{equation}
	The identity 
	\[
	\frac{z-\zeta_1}{z-\bar \zeta_1}- \frac{z-\zeta_2}{z-\bar \zeta_2}
	= \frac{\zeta_2-\zeta_1}{z-\bar \zeta_1} 
	+ \frac{\bar \zeta_2-\bar\zeta_1}{z-\bar \zeta_1}\frac{z-\zeta_2}{z-\bar \zeta_2}
	\]
	implies the following inequality for arbitrary $z,\zeta_1,\zeta_2\in\bC_+$: 
	\[
	\left|\frac{z-\zeta_1}{z-\bar \zeta_1}- \frac{z-\zeta_2}{z-\bar \zeta_2}\right|
	\le \frac{2|\zeta_1-\zeta_2|}{|\myIm z|}.
	\]
	Therefore, for every $n\in\bN$ and every $z\in\bC_+$,
	\[
	|B(z,\bol^{(n)})-B(z,\bol^{(n+1)})| \le \frac{2|\lambda_n-\mu_n|}{|\myIm z|},
	\]  
	which on account of~\eqref{eq.37} results in~\eqref{eq.35}.
\end{proof}

\begin{lemma}\label{le.51b} Assume that $(f_n)_{n\in\bN}$ is a sequence of non-negative functions in $L_1(\bR)$ converging to $f\in L_1(\bR)$ almost everywhere. Assume also that there exists a finite limit $\lim\limits_{n\to\infty} \|f_n\|_{L_1}$ and that, for every $\varepsilon>0$, the sequence $(f_n)_{n\in\bN}$ converges to~$f$ in the space~$L_1(\bR_{\varepsilon})$. Then 
	\begin{equation*}
	\lim_{n\to\infty} \int_{-\infty}^\infty \frac{f_n(x)}{z-x} \, dx 
	= \frac{\gamma}{z} +  \int_{-\infty}^\infty \frac{f(x)}{z-x} \, dx, 
	\qquad z \in\bC_+,
	\end{equation*}
	where $\gamma:=\lim\limits_{n\to\infty} \|f_n\|_{L_1} \, - \,\|f\|_{L_1}$ is non-negative.
\end{lemma}

\begin{proof}
	The fact that the number~$\gamma$ is non-negative follows from Fatou's lemma. Next, for every $k\in\bN$ there exists an $N_k \in\bN$ with the property that 
	\[
	\int_{|x|\ge 1/k} |f_n(x)-f(x)| \, dx \le \tfrac1k
	\]
	as soon as $n\ge N_k$. Assuming that $N_k$ strictly increase and setting $\varepsilon_n:= \tfrac1k$ for $n \in [N_k, N_{k+1})$, we see that $\varepsilon_n\to 0$ as $n\to\infty$ and that
	\[
	\lim_{n\to\infty}\int_{|x|\ge \varepsilon_n} |f_n(x)-f(x)| \, dx =0.
	\]
	Set $\Delta_n:=(-\varepsilon_n,\varepsilon_n)$ and $\Delta'_n:=\bR\setminus\Delta_n$.
	Then for every $z\in\bC_+$ we get
	\[
	\int_{\Delta_n} \frac{f_n(x)}{z-x} \, dx
	=\frac1z \int_{\Delta_n} f_n(x) \, dx  +o(1)
	=\frac{\gamma}{z}+o(1), \qquad n\to \infty,
	\]
	and 
	\[
	\lim\limits_{n\to\infty}\int_{\Delta'_n} \frac{f_n(x)}{z-x} \, dx
	=\int_{-\infty}^\infty \frac{f(x)}{z-x} \, dx
	\]
	yielding the statement of the lemma.
\end{proof}

\section{Proof of the main result}\label{sec:proof}

We recall first some standard facts about the relations between the coefficient~$a$ and the negative spectrum of the operator~$T_q$ (see e.g.,~\cite[Ch~3.5]{Mar}). For a potential~$q\in \cQ_1$ and any nonzero $k\in \bR$, the Jost solutions $e_-(\cdot,k;q)$ and $e_-(\cdot,-k;q)$ form a fundamental system of solutions of the equation $-f''+qf=k^2f$. Comparing the asymptotics of both parts of~\eqref{eq.211} at~$-\infty$ shows that
\[
e_+(x,k;q) = a(k;q) e_-(x,-k;q) + b(k;q) e_-(x,k;q)
\]
and, therefore, the Wronskian 
\[
W(e_+(x,k;q),e_-(x,k;q)):= e'_+(x,k;q)e_-(x,k;q) - e_+(x,k;q)e'_-(x,k;q)
\] 
of $e_+(\,\cdot\,,k;q)$ and $e_-(\,\cdot\,,k;q)$ is equal to
\[
a(k;q) W(e_-(x,-k;q),e_-(x,k;q)).
\]
As the Wronskian does not depend on~$x$, letting $x$ to $-\infty$ and using the asymptotics of the Jost solution $e_-$ and its derivative (see Corollary~\ref{cor:Jost.e}) we find that
\[
W(e_+(x,k;q),e_-(x,k;q)) = 2ik a(k;q).
\]

Observe now that both parts of the above equality admit continuation to functions that are analytic and bounded in $\bC_+$ and continuous in~$\overline{\bC_+}$; therefore, the above equality continues to hold in~$\overline{\bC_+}$. It follows that if $z$ is a zero of~$a(\cdot;q)$ in $\bC_+$, then the Jost solutions~$e_+(\cdot,z;q)$ and $e_-(\cdot,z;q)$ are linearly dependent and thus they exponentially decay as $|x|\to \infty$. As a result, $z^2$ is an eigenvalue of the Schr\"odinger operator~$T_q$ with eigenfunction~$e_\pm(\cdot;z;q)$ and thus $z^2$ is one of $-\kappa^2(j;q)$. 

Conversely, to each eigenvalue $-\kappa^2(j;q)<0$ of $T_q$ there corresponds an eigenfunction, i.e., a solution of~\eqref{eq:Jost.eq} with $\lambda^2 = -\kappa^2(j;q)$ that is square integrable on the whole real line. As the Schr\"odinger operator~$T_q$ with integrable potential~$q$ is in the limit point case at $\pm\infty$, the only (up to a constant factor) solution of the equation $-y'' + qy = -\kappa_j^2y$ that is square integrable at $\pm\infty$ is the Jost solution~$e_\pm (\cdot,i\kappa(j;q);q)$. Therefore, the right Jost solution $e_+ (\cdot,i\kappa_j;q)$ and the left Jost solution~$e_- (\cdot,i\kappa(j;q);q)$ must be collinear and thus $a(i\kappa(j;q);q)=0$. 

The above reasoning identifies the numbers~$i\kappa(j;q)$ for which $\kappa(j;q)>0$ as zeros of the coefficient~$a(\cdot;q)$. 
Assume now that $q_0\in \cQ_1$ and that $j\in\bN$ is such that $\kappa(j;q_0)>0$. Then continuity of the 
mapping~$L_1(\bR)\ni q \mapsto a(\cdot;q)\in \cA$ (cf.~Lemma~\ref{lem:Jost.a}) along with Rouche's theorem imply that the mapping
$\cQ_1\ni q \mapsto \kappa(j;q)\in \cA$ 
is continuous in some small neighbourhood of the point~$q_0$. We next prove that the whole sequence~$\kappa(\cdot;q)$ depends continuously on~$q$ in suitable topology. 

\begin{lemma}\label{lem:4.2} The mapping
	$\cQ_1\ni q\mapsto \kappa(\cdot;q)\in\ell_1(\bN)$ 
	is continuous. 
\end{lemma}

\begin{proof} Consider an arbitrary sequence $(q_n)_{n\in\bN}$ of potentials from $\cQ_1$ that converges to some $q\in\cQ_1$ in the $L_1$-norm; we then prove that  $\kappa(\cdot;q_n)\to\kappa(\cdot;q)$ as $n\to\infty$ in the $\ell_1$-norm. 
	Without loss of generality, we can assume that 
	\begin{equation} \label{eq:4.2}
	\|q-q_n\|_{L_1}\le 2^{-n}, \qquad n\in\bN,
	\end{equation}
	and introduce auxiliary sequences
	\[     
	q^-_n:= q - \sum\limits_{k=n}^\infty|q-q_k|,  \qquad
	q^+_n:= q + \sum\limits_{k=n}^\infty|q-q_k|,  \qquad   n\in\bN.
	\]
	In virtue of~\eqref{eq:4.2}, they both converge to~$q$ in the topology of~$\cQ_1$ and, moreover,
	\[
	q^-_n\le  q^-_{n+1}\le  q\le q^+_{n+1}\le q^+_n,  \qquad  q^-_n\le  q_n\le  q^+_n, \qquad n\in\bN.
	\]
	By the minimax principle~\cite[Ch.~XIII.1]{ReeSim}, for every $j\in\bN$ and $n\in\bN$ we get
	\[
	\kappa(j,q^+_n) \le \kappa(j,q^+_{n+1}) 
	\le \kappa(j,q)
	\le \kappa(j,q^-_{n+1})\le \kappa(j,q^-_n)
	\]
	and 
	\begin{equation}\label{eq:4.kappa-minimax}
	\kappa(j,q^+_n)\le \kappa(j,q_n) \le \kappa(j,q^-_{n});
	\end{equation}
	moreover, the continuity of each separate~$\kappa(j;q)$ implies
	\[  
	\lim_{n\to\infty}|\kappa(j,q)- \kappa(j,q^\pm_n)|=0.
	\]
	The Lebesgue dominated convergence theorem now yields
	\[
	\lim_{n\to\infty} \|\kappa(\cdot,q)- \kappa(\cdot,q^\pm_n)\|_1=0
	\]
	which, in view of the estimates (cf.~\eqref{eq:4.kappa-minimax})
	\begin{align*}
	|\kappa(j,q)- \kappa(j,q_n)| &\le \max\{|\kappa(j,q)- \kappa(j,q^+_n)|,|\kappa(j,q)- \kappa(j,q^-_n)|\} \\
	&\le |\kappa(j,q)- \kappa(j,q^+_n)| +|\kappa(j,q)- \kappa(j,q^-_n)|
	\end{align*}
	finishes the proof.
\end{proof}

For $q\in\cQ_1$, we denote by 
\[
B(z;q) := \prod_{j=1}^{\infty} \frac{z-i\kappa(j;q)}{z+i\kappa(j;q)}
\]  
the Blaschke product corresponding to the sequence $(i\kappa(j;q))$ (where, as usual, $\kappa(j;q)=0$ for $j$ larger than the number of negative eigenvalues of the operator~$T_q$). 
It is known~\cite[Ch.~3.5]{Mar} that if $q\in\cQ_{1,1}$ is a potential from the Faddeev--Marchenko class, then the values in~$\bC_+$ of the scattering coefficient~$a(\cdot;q)$ can be uniquely determined from its values on the real line via the Poisson--Schwarz formula
\begin{equation} \label{eq:proof.PS}
a(z;q) 
= B(z;q) \exp\Bigl\{\frac{1}{\pi i} \int_{\bR} \frac{\log |a(k;q)|}{k-z}\,dk \Bigr\}.	
\end{equation}
It also follows from the results of Gesztesy and Holden~\cite{GH} that for all $q\in\cQ_{1,1}$ the first trace formula~\eqref{eq:intr.FZ} holds true. Since (see Remark~\ref{re.25})
\begin{equation}\label{eq.45}
(1-|r(k;q)|^2)^{-1}=|a(k;q)|^2, \qquad   k\in\bR\setminus\{0\}, \quad q\in\cQ_1,
\end{equation}
that trace formula can be written in the form 
\begin{equation}\label{eq.46}
-4\sum_{j=1}^{\infty}\kappa(j;q) + \frac{2}{\pi}\int_{\bR}\log |a(k;q)|\,dk 
= \int_{\bR} q(x)\,dx, \qquad q\in\cQ_{1,1}.
\end{equation}
Formula~\eqref{eq.46} is our starting point in proving the main result of the paper.

\begin{proof}[Proof of Theorem~\ref{thm:main}.]
	Continuity of the first mapping in~\eqref{eq:intr.cont} is established in Lemma~\ref{lem:4.2}. 
	We next prove that the trace formula \eqref{eq.46} holds for all potentials in~$\cQ_1$. The idea is to extend that formula by continuity starting from potentials in the  Faddeev--Marchenko class~$\cQ_{1,1}$, which are dense in~$\cQ_1$ in the $L_1$-norm.
	
	To this end, we fix any $q\in\cQ_1$ and choose a sequence of potentials $q_n\in\cQ_{1,1}$, $n\ge1$, converging to $q$ in the topology of $L_1(\bR)$. Next we rewrite formula~\eqref{eq.46} for $q_n$ in the form
	\begin{equation*}\label{eq.47}
	\int_{\bR}\log |a(k;q_n)|\,dk
	= \frac{\pi}{2}\int_{\bR} q_n(x)\,dx + 2\pi\sum_{j=1}^{\infty} \kappa(j;q_n),
	\end{equation*}
	and observe that, by Lemma~\ref{lem:4.2}, the right-hand side above possesses a finite limit as~$n\to\infty$ and thus
	\[
	\lim_{n\to\infty}\int_{\bR}\log |a(k;q_n)|\,dk 
	= \frac{\pi}{2}\int\limits_{\bR} q(x)\,dx 
	+ 2\pi\sum\limits_{j=1}^{\infty}\kappa(j;q).
	\]
	Since the functions $\log |a(\,\cdot\,;q_n)|$ are non-negative on $\mathbb{R}\setminus\{0\}$ and converge pointwise therein to the non-negative function $\log |a(\,\cdot\,;q)|$ as $n\to\infty$, Fatou's lemma yields the inequality 
	\begin{equation}\label{eq.48}
	\gamma
	:= \lim_{n\to\infty}\int_{\bR}\log |a(k;q_n)|\,dk - \int_{\bR}\log |a(k;q)|\,dk \ge0.
	\end{equation}
	In particular, the function $\log|a(\,\cdot\,;q)|$ is integrable on the line, and it remains to show that $\gamma=0$.

	By Lemma~\ref{lem:4.2}, $\kappa(\cdot;q_n)\to \kappa(\cdot;q)$ in~$\ell_1$-norm as $n\to \infty$, and then Lemma~\ref{le.32} guarantees convergence of the corresponding Blaschke products, i.e., that $B(z;q_n)\to B(z;q)$ for every $z\in\bC_+$. Next, Lemma~\ref{le.51b} implies that, for every $z\in \bC_+$, 
	\begin{equation*} 
	\lim_{n\to\infty} \exp\Bigl\{\frac{1}{\pi i} 
	\int_{\bR} \frac{\log{|a(k; q_n)|}}{k-z}\,dk \Bigr\}
	= e^{i\gamma/\pi z}\exp\Bigl\{ \frac{1}{\pi i} 
	\int_{\bR} \frac{\log{|a(k; q)|}}{k-z}\,dk \Bigr\}
	\end{equation*}
	with $\gamma$ defined in~\eqref{eq.48}.    
	Finally, for every $z\in\bC_+$ we have $a(z;q_n) \to a(z;q) $
	as $n\to\infty$ by virtue of Lemma~\ref{lem:Jost.a}. Combining these facts and letting $n\to\infty$ in equation~\eqref{eq:proof.PS} for $q=q_n$, we arrive at the equality
	\begin{equation} \label{eq:412}
	a(z;q) = e^{i\gamma/\pi z} B(z;q) \exp\Bigl\{ \frac{1}{\pi i} 
	\int_{\bR} \frac{\log|a(k,q)|}{k-z}\,dk \Bigr\}, 
	\qquad z\in\bC_+.
	\end{equation}
	
	Since $\log |a(k;q)|\ge 0$ for $k\in \bR\setminus\{0\}$, we conclude that
	\[
	\Bigl|\exp\Bigl\{ \frac{1}{\pi i} \int_{\bR} \frac{\log|a(k;q)|}{k-z}\,dk \Bigr\}\Bigr|\ge 1, 
	\qquad z\in \bC_+,
	\]
	so that \eqref{eq:412} results in the relation
	\[
	\log^+|a(z;q)|\ge \log|e^{i\gamma/\pi z}| +\log |B(z;q)|, \qquad z\in \bC_+.
	\]
	Setting  $z=\xi e^{it}$ with $\xi>0$ and $t\in (0, \pi)$, we arrive at the inequality
	\begin{equation*}
	\frac{\gamma}{\pi} \sin t
	\le \xi\log^+|a(\xi e^{it};q)| +\xi\log|B(\xi e^{it};q)|^{-1},
	\end{equation*}
	which upon integration in $t$ produces
	\begin{equation*} 
	\frac{2\gamma}{\pi} \le \int_0^\pi \xi\log^+|a(\xi e^{it};q)|\, dt 
	+ \int_0^\pi \xi\log|B(\xi e^{it};q)|^{-1}\, dt.
	\end{equation*}
	Letting $\xi\to+0$ in the above relation, on account of~\eqref{eq:jost.a-asympt} and \eqref{eq.32} one concludes that $\gamma\le 0$, so that $\gamma=0$ and thus the trace formula~\eqref{eq:intr.FZ} holds for all~$q\in\cQ_1$.
	
	It remains to prove continuity of the mapping $\cQ_1\ni q\mapsto h(\cdot;q)\in  \cQ_1$. Assume that a sequence $q_n \in \cQ_1$ converges to a $q\in\cQ_1$ in the $L_1$-norm; we have to prove that $h(\cdot;q_n)\to h(\cdot;q)$ in $L_1(\bR)$ as $n\to\infty$. Recall that, in view of~\eqref{eq.45}, we have 
	\[
	h(k;q)=2\log |a(k;q)|, \quad f(k;q_n) = 2\log |a(k;q_n)|, \qquad k\in\bR_0, \quad n\in\bN.
	\]
	Passing to the limit in the trace formula~\eqref{eq.46} for $q_n$, we conclude that
	\[
	\lim_{n\to\infty} \int_{\bR} h(k;q_n)\,dk = \int_{\bR} h(k;q)\,dk.
	\]
	As $h(\,\cdot\,;q_n)$ and $h(\,\cdot\,;q)$ are non-negative, this means that $\|h(\,\cdot\,;q_n)\|_1 \to \|h(\,\cdot\,;q)\|_1$ as $n\to\infty$. Since also $h(k;q_n)\to h(k;q)$ as $n\to\infty$ for all $k\in\bR_0$ by Remark~\ref{re.25}, the statement proved by F.~Riesz~\cite{Rie} (widely known as Scheff\'e's lemma~\cite{Sch,Kus}) establishes $L_1$-convergence of~$h(\cdot;q_n)$ to $h(\cdot;q)$ and thus completes the proof of the theorem.
\end{proof}

\medskip\noindent\emph{Acknowledgement}
	The authors thank Prof.~Gesztesy for stimulating discussions and literature comments. The research was supported by Ministry of Education and Science of Ukraine, grant no.~0118U002060. The first author acknowledges support of the Centre for Innovation and Transfer of Natural Sciences and Engineering Knowledge at the University of Rzesz\'ow.


\end{document}